\documentclass{patmorin}
\usepackage{pat}
\usepackage[T1]{fontenc}
\usepackage[utf8]{inputenc}
\usepackage{mathtools}
\usepackage{thmtools}
\usepackage{thm-restate}

\usepackage[inline]{enumitem}
\usepackage{soul}
\usepackage[normalem]{ulem}

\crefname{p}{}{}
\creflabelformat{p}{#2(#1)#3}

\usepackage{todonotes}
\newenvironment{clmproof}{\begin{proof}[Proof of Claim:]}{\end{proof}}

\usepackage[longnamesfirst,numbers,sort&compress]{natbib}

\DeclareMathOperator{\sep}{sn}
\DeclareMathOperator{\tw}{tw}
\DeclareMathOperator{\width}{width}

\DeclareMathOperator{\depth}{depth}
\DeclareMathOperator{\height}{height}

\DeclareMathOperator{\interior}{int}
\DeclareMathOperator{\boundary}{\partial}

\title{\MakeUppercase{Separation Number and Treewidth, Revisited}%
  \thanks{This research was partly funded by NSERC.}}

\author{Hussein Houdrouge%
  \thanks{School of Computer Science, Carleton University.},\quad
  Babak Miraftab\footnotemark[2],\quad
  and Pat Morin\footnotemark[2]}

\date{}

\begin{document}

\maketitle

\begin{abstract}
  We give a constructive proof of the fact that the treewidth of a graph $G$ is bounded by a linear function of the separation number of $G$.
\end{abstract}

\section{Introduction}

In this paper every graph $G$ is undirected and simple with vertex set $V(G)$ and edge set $E(G)$.  A \defin{tree decomposition} of a graph $G$ is a collection $\mathcal{T}:=(B_x:x\in V(T))$ of vertex subsets of $G$, called \defin{bags}, that is indexed by the vertices of a tree $T$ and such that
\begin{enumerate*}[label=(\roman*)]
  \item\label{covers_edges} for each $vw\in E(G)$, there exists $x\in V(T)$ such that $\{v,w\}\subseteq B_x$; and
  \item\label{connectivity} for each vertex $v$ of $G$, $T[\{x\in V(T): v\in B_x\}]$ is a non-empty (connected) subtree of $T$.
\end{enumerate*}
The \defin{width}, of a tree decomposition $\mathcal{T}:=(B_x:x\in V(T))$ is $\mathdefin{\width(\mathcal{T})}:=\max\{|B_x|:x\in V(T)\}-1$. The \defin{treewidth} of a graph $G$ is $\mathdefin{\tw(G)}:=\min\{\width(\mathcal{T}):\text{$\mathcal{T}$ is a tree decomposition of $G$}\}$.

A \defin{separation} $(A,B)$ of a graph $G$ is a pair of subsets of $V(G)$ with $A\cup B= V(G)$ and such that, for each edge $vw$ of $G$, $\{v,w\}\subseteq A$ or $\{v,w\}\subseteq B$.  The \defin{order} of a separation $(A,B)$ is $|A\cap B|$.  A separation $(A,B)$ is \defin{balanced} if $|A\setminus B|\le \tfrac{2}{3}|V(G)|$ and $|B\setminus A|\le \tfrac{2}{3}|V(G)|$.  The \defin{separation number} $\mathdefin{\sep(G)}$ of a graph $G$ is the minimum integer $a$ such that every subgraph of $G$ has a balanced separation of order at most $a$.

A short argument due to \citet{robertson.seymour:graph}, which has many generalizations, shows that for any graph $G$, $\sep(G)\le \tw(G)+1$ .  We reprove a weak converse of this fact, first proven by \citet{dvorak.norin:treewidth}.
\begin{thm}\label{main_result}
  There exists a constant $c$ such that, for every graph $G$, $\tw(G)\le c\cdot \sep(G)$.
\end{thm}

To put \cref{main_result} into context, consider the following classic result of \citet{robertson.seymour:graph}.
Let $W$ be a subset of the vertices in a graph $G$.  A separation $(A,B)$ of $G$ is \defin{$W$-balanced} if $|(A\setminus B)\cap W|\le \tfrac{2}{3}|W|$ and $|(B\setminus A)\cap W|\le \tfrac{2}{3}|W|$.

\begin{thm}[\citet{robertson.seymour:graph}]\label{w_balanced_theorem}
    Let $G$ be a graph with the property that, for each $W\subseteq V(G)$, $G$ has a $W$-balanced separation of order at most $a$. Then  $\tw(G)<4a$.
\end{thm}

The proof of \cref{w_balanced_theorem} is constructive and intuitive.  Indeed, a tree decomposition of $G$ can be constructed iteratively by an algorithm that maintains a separation $(X,Y)$ of order at most $3a$, and a tree decomposition $\mathcal{T}:=(B_x:x\in V(T))$ of $G[Y]$ of width less than $4a$ in which some bag $B_x$ contains all vertices in $W:=X\cap Y$.  To extend $\mathcal{T}$, the algorithm takes a $W$-balanced separation $(A,B)$ of $G[X]$ of order at most $a$ and creates a leaf $x'$ in $T$ adjacent to $x$ with $B_{x'}:=W\cup (A\cap B)$.\footnote{A $W$-balanced separation $(A,B)$ of $G[X]$ can be obtained from a $W$-balanced separation $(A',B')$ of $G$ by setting $(A,B):=(A'\cap X,B'\cap X)$.}  Note that $|B_{x'}|\le |W|+|A\cap B|\le 4a$, so the width of $\mathcal{T}$ is still less than $4a$.
Let $X_A:=A$, $Y_A:=Y\cup(A\cap B)$, $G_A:=G[X_A\cup Y_A]=G[Y\cup A]$ and $W_A:=X_A\cap Y_A$.  Then $\mathcal{T}$ is a tree decomposition $G[Y_A]$ of width less than $4a$ and $(X_A,Y_A)$ is a separation of $G_A$ of order $|W_A|\le|W\setminus B|+|A\cap B|\le \tfrac{2}{3}|W|+a\le 3a$ and a bag $B_{x'}$ of $\mathcal{T}$ contains $W_A$. The algorithm then inductively extends $\mathcal{T}$ to a tree decomposition of $G_A$ of width less than $4a$.  Let $X_B:=B$, $Y_B:=Y\cup A$, $G_B:=G$, and $W_B:=X_B\cap Y_B$.  Then $\mathcal{T}$ is a tree decomposition of $Y_B$ of width less than $4a$ and $(X_B,Y_B)$ is a separation of $G_B$ of order $|W_B|\le |W\cap B|+|A\cap B|\le\tfrac{2}{3}|W|+a\le 3a$ and a bag $B_{x'}$ of $\mathcal{T}$ contains $W_B$.  The algorithm finishes by inductively extending $\mathcal{T}$ to a tree decomposition of $G_B=G$ of width less than $4a$.

The challenge in establishing \cref{main_result} is that the balanced separations in the definition of separation number are only guaranteed to balance the entire set of vertices in an arbitrary subgraph of $G$, rather than separating $G$ in such a way the vertices in a specific $W\subseteq V(G)$ are balanced.  In the language of the previous paragraph, there is no reason that a balanced separation $(A,B)$ of $G[X]$ should have $|(A\setminus B)\cap W|<|W|$ and $|(B\setminus A)\cap W|<|W|$.

\citet{dvorak.norin:treewidth} prove \cref{main_result} with the constant $c=15$. Their proof is by contradiction and makes use of the relationship between treewidth and brambles established by \citet{seymour.thomas:graph}.  Essentially, they show that if $\tw(G)>15\sep(G)$, then there exists an $\alpha$-tame $W$-cloud (a special kind of network flow) which contradicts the choice of a haven (a special kind of flap assignment) that is derived from a bramble of order $15\sep(G)$.  The crux of their proof \cite[Proof of Lemma~7]{dvorak.norin:treewidth} involves showing that, for a \emph{carefully chosen} $W\subseteq V(G)$, a balanced separation of the subgraph $H\subseteq G$ induced by the saturated and hungry vertices of an $\alpha$-tame $W$-cloud is, by necessity, also (rougly) $W$-balanced. This leads to a contradiction related to the choice of $W$.\footnote{In an earlier draft of their result, \citet{dvorak.norin:treewidth_v1}, used tangles rather than brambles and havens, and confluent flows \cite{chen.kleinberg.ea:almost} rather than $W$-clouds to establish \cref{main_result} with the constant $c=105$.  They credit an anonymous referee for help in reducing the value of $c$.}

In the current paper, we prove \cref{main_result} with the constant $c=7915/139< 56.943$.  Despite the larger constant, we believe that the proof given here has a number of advantages.  The proof is constructive: It proves that $\tw(G)\le c\cdot \sep(G)$ by constructing a tree decomposition of $G$ having width less than $c\cdot\sep(G)$.  The proof requires fewer definitions and previous results: It does not use brambles, havens, or network flows.  Brambles and havens are avoided entirely.  The use of network flows is replaced by a collection of paths obtained from repeated applications of the simplest version of Menger's Theorem on vertex-disjoint paths in (unweighted undirected) graphs.

Most importantly, the proof given here is built around a generalization of $W$-balanced separations:  For a sufficiently large $t>0$ and an \emph{arbitrary} $W\subseteq V(G)$ of size at least $t\cdot \sep(G)$, there exists a subgraph $H$ of $G$ with $W\subseteq V(H)$ and a separation $(X,Y)$ of $G$ with $W\subseteq Y$ and $V(G)\setminus V(H)\subseteq X\setminus Y$ and having order less than $|W|$ and such that any balanced separation of $H$ must necessarily balance $W\cup (X\cap Y)$.  This leads to an algorithm for constructing a tree decomposition of $G$ similar in spirit to the algorithm outlined above.
In particular, recursively taking balanced separations of subgraphs of $H$ gives an algorithm for constructing a tree decomposition $\mathcal{T}_Y:=(B_x:x\in V(T_Y))$ of $G[Y]$ in which some bag $B_y$ contains $W\cup (X\cap Y)$.  Then, recursion/induction is used to find a tree decomposition $\mathcal{T}_X:=(B_x:x\in V(T_X))$ of $G':=G[X]$ in which some bag $B_x$ contains $W':=X\cap Y$.  Joining $T_X$ and $T_Y$ with the edge $xy$ gives a tree $T$ such that $\mathcal{T}:=(B_x:x\in V(T))$ is the desired tree decomposition of $G$.

\section{Preliminaries}

For standard graph theoretic terminology and notations, see \citet{diestel2017graph}.
Let $G$ be a graph and let $S,T\subseteq V(G)$. An \defin{$S$-$T$ path} in $G$ is a path in $G$ whose first vertex is in $S$ and whose last vertex is in $T$.  We say that a set $Z\subseteq V(G)$ \defin{separates} $S$ and $T$ if $G-X$ has no $S$-$T$ path.  When $Z$ separates $S$ and $T$, any separation $(X,Y)$ of $G$ with $S\subseteq X$, $T\subseteq Y$ and $X\cap Y=Z$ is called an \defin{$(S,Z,T)$-separation}. To see that an $(S,Z,T)$-separation always exists, let $G_X$ be the union of all components of $G-Z$ that contain a vertex of $S$. Then $S\subseteq V(G_X)\cup Z$. Take $G_Y$ to be the union of all components of $G-Z$ not included in $G_X$. Since  every component of $G_X$ contains a vertex in $S$, no component of $G_X$ contains a vertex in $T$.  Therefore, $T\subseteq V(G_Y)\cup Z$.  Then the separation $(X,Y):=(V(G_X)\cup Z, V(G_Y)\cup Z)$ is an $(S,Z,T)$-separation.

We make use of the following vertex connectivity version of Menger's Theorem (see, for example, \citet[Theorem~3.3.1]{diestel2017graph}): \

\begin{thm}[Menger's Theorem]\label{menger}
  Let $G$ be a graph and let $S$ and $T$ be subsets of $V(G)$. For each $k\in\N$, exactly one of the following is true:
  \begin{enumerate}[nosep,nolistsep,label=\rm(\roman*),ref=(\roman*)2]
      \item $G$ contains $k$ pairwise vertex-disjoint $S$-$T$ paths; or
      \item $G$ has a vertex subset $Z$ of size less than $k$ that separates $S$ and $T$.
  \end{enumerate}
\end{thm}

The \defin{depth}, $\mathdefin{\depth_T(x)}$ of a node $x$ in a rooted tree $T$ is the number of edges on the path from $x$ to the root of $T$. The \defin{height} of a rooted tree $T$ is $\mathdefin{\height(T)}:=\max\{\depth_T(x):x\in V(T)\}$. For a node $x$ in a rooted tree $T$, we let $\mathdefin{T_x}$ denote the subtree of $T$ induced by all the descendants of $x$, including $x$ itself.

A tree decomposition $\mathcal{T}:=(B_x:x\in V(T))$ of a graph $G$ is \defin{rooted} if the tree $T$ is rooted. If $x_0$ is the root of $T$, then $B_{x_0}$ is called the \defin{root bag} of $\mathcal{T}$.  Let $\mathcal{T}:=(B_x:x\in V(T))$ be a rooted tree decomposition of a graph $G$ where $x_0$ is the root of $T$.   The \defin{boundary} of $x_0$ is $\mathdefin{\boundary_{\mathcal{T}}(x_0)}:=\emptyset$.  For a node $x$ of $T$ with parent $y$, the \defin{boundary} of $x$ is $\mathdefin{\boundary_\mathcal{T}(x)}:=B_x\cap B_y$.  The \defin{interior} of a node $x$ in $T$ is $\mathdefin{\interior_\mathcal{T}(x)}:=(\bigcup_{x'\in V(T_x)} B_{x'})\setminus \boundary(x)$.  Note that, for the root $x_0$ of $T$, $\interior_{\mathcal{T}}(x_0)=V(G)$.  From these definitions it follows that, if $T_y\supseteq T_x$, then $\interior_{\mathcal{T}}(y)\supseteq\interior_{\mathcal{T}}(x)$ and that $\interior_\mathcal{T}(y) \cup \boundary_\mathcal{T}(y)\supseteq \interior_\mathcal{T}(x) \cup \boundary_\mathcal{T}(x)$. In particular, these inclusion relations hold for every ancestor $y$ of $x$.

The following lemma allows us to restrict a rooted tree decomposition of a graph $G$ to the subgraph $G[Y]$ induced by one part of a separation $(X,Y)$ in such a way that $X\cap Y$ is contained in a single bag (the root bag) of the resulting decomposition.

\begin{lem}\label{restricted_decomp}
    Let $\mathcal{T}':=(B'_x:x\in V(T'))$ be a rooted tree decomposition of a graph $G$, let $(X,Y)$ be a separation of $G$, and let $B_x:=(B'_x\cap Y)\cup (\interior_{\mathcal{T}'}(x) \cap X \cap Y)$ for each $x\in V(T')$.  Then $\mathcal{T}:=(B_x:x\in V(T'))$ is a tree decomposition of $G[Y]$.  Furthermore, the root bag $B_{x_0}$ of $\mathcal{T}$ contains $X\cap Y$.
\end{lem}

\begin{proof}
  The ``furthermore'' clause of the statement is immediate, since $\interior_{\mathcal{T}}(x_0)=V(G)$, so $B_{x_0}\subseteq V(G)\cap X\cap Y=X\cap Y$.
  To show that $\mathcal{T}$ is a tree decomposition of $G$ we must show that $\mathcal{T}$ satisfies Properties~\ref{covers_edges} and \ref{connectivity} of tree decompositions.
  Let $vw$ be an edge of $G[Y]$.  Since $vw$ is also an edge of $G$ and  $\mathcal{T'}$ is a tree decomposition of $G$,  $\{v,w\}\subseteq B'_x$ for some $x\in V(T')$, so $\{v,w\}\subseteq B'_x\cap Y\subseteq  B_x$. Thus, $\mathcal{T}$ has Property~\ref{covers_edges} of tree decompositions.

  Suppose, for the sake of contradiction, that $\mathcal{T}$ violates Property~\ref{connectivity}.  Then there exists some $v\in Y$, some $r\ge 2$, and some path $P=x_0,\ldots,x_r$ in $T'$ with $v\in B_{x_0}$, $v\in B_{x_r}$ and $v\not\in B_{x_i}$ for some $i\in\{1,\ldots,r-1\}$.  Let $P$ be chosen so that its length, $r$, is minimum.  Then $v\not\in B_{x_i}$ for each $i\in\{1,\ldots,r-1\}$.  Since $\mathcal{T}'$ is a tree decomposition of $G$, $v\not\in B'_{x_0}$ or $v\not\in B'_{x_r}$ since, otherwise $v\in B_{x_i}\subseteq B'_{x_i}\cap Y$ for each $i\in\{0,\ldots,r\}$.  Assume, without loss of generality that $v\not\in B'_{x_r}$. Then $v\in\interior_{\mathcal{T}'}(x_r)\cap X\cap Y$ since $v\in B_{x_r}$.  In particular, $v\in X\cap Y$.

  Define $i^*$ so that $x_{i^*}$ is the unique vertex in $P$ that is an ancestor of both $x_0$ and $x_r$.  Then  $\interior_{\mathcal{T}'}(x_{i^*})\supseteq \interior_{\mathcal{T}'}(x_i)$ for each $i\in\{0,\ldots,r\}$.   In particular $v\in \interior_{\mathcal{T}'}(x_r)\cap X\cap Y\subseteq\interior_{\mathcal{T}'}(x_{i^*})\cap X\cap Y\subseteq B_{x_{i}}$ for each $i\in\{i^*,\ldots,r\}$.  Therefore $i^*=r$ and $x_r$ is an ancestor of $x_0$.  Therefore $x_1$ is the parent of $x_0$, $v\in B_{x_0}$ and $v\not\in B_{x_1}$.  If $v\in\interior_{\mathcal{T}'}(x_0)$ then $v\in\interior_{\mathcal{T}'}(x_0)\cap X\cap Y\subseteq\interior_{\mathcal{T}'}(x_1)\cap X\cap Y\subseteq B_{x_1}$, a contradiction.  Therefore $v\in B'_{x_0}$ since $v\in B_{x_0}$.  Thus $x_1$ is the parent of $x_0$, $v\in B'_{x_0}$ and $v\not\in B'_{x_1}$.  Therefore $v\in\interior_{\mathcal{T}'}(x_1)$, so $v\in\interior_{\mathcal{T}'}(x_1)\cap X\cap Y\subseteq B_{x_1}$, also a contradiction.
\end{proof}

The following construction of a tree decomposition using balanced separations (or variants of this construction using balanced separators) is fairly standard.

\begin{lem}\label{separation_tree}
  Let $G$ be an $n$-vertex graph with $\sep(G)\le a$.  Then, for every integer $h\ge 0$, $G$ has a rooted tree decomposition $\mathcal{T}:=(B_x:x\in V(T))$ such that
  \begin{enumerate}[nosep,nolistsep,label=(\roman*)]
    \item\label{height_bound} $\height(T)\le h$;
    \item\label{size_bounds} for each $x\in V(T)$,  $|\interior_\mathcal{T}(x)|\le n\cdot(\tfrac{2}{3})^{\depth_T(x)}$ and $|\boundary(x)|\le \depth_T(x)\cdot a$;
    \item\label{leaf_size_bounds} for each leaf $y$ of $T$, $|\interior_{\mathcal{T}}(y)|\le n\cdot(\tfrac{2}{3})^h$.
  \end{enumerate}
\end{lem}

\begin{proof}
  The tree decomposition $\mathcal{T}=(B_x:x\in V(T))$ and its supporting tree $T$ is constructed recursively, as follows:  Fix a global value $N:=n\cdot(\tfrac{2}{3})^h$ that does not change during recursive invocations. Each recursive invocation takes a pair $(G',\partial')$ and the initial invocation is on the pair $(G,\emptyset)$. When recursing on $(G',\partial')$ to construct a subtree $T'$ we apply the following rule:  If $|V(G')\setminus\partial'|\le N$, then $T'$ consists of a single node $x$ with $B_{x}:=V(G')$.  Otherwise, let $(A',B')$ be a balanced separation of $G'-\partial'$ of order at most $a$.  The root $x$ of $T'$ has $B_{x}:=\partial'\cup (A'\cap B')$. The left child of $x$ is the root of the tree obtained by recursing on $(G'[A'],A'\cap(\partial'\cup B'))$ and the right child of $x$ is the root of the tree obtained by recursing on $(G[B'],B'\cap(\partial'\cup A'))$.

  We now show that $\boundary_{\mathcal{T}}(x)=\partial_x$, for each subtree $T_x$ rooted at $x\in V(T)$ that was constructed by a recursive invocation on $(G_x,\partial_x)$. If $\depth_T(x)=0$ then $x$ is the root of $T$ and $\boundary_{\mathcal{T}}(x)=\emptyset=\partial_x$, by definition.
  Now suppose $\depth_T(x)\ge 1$, the parent of $x$ is $y$ and $T_y$ is the result of a recursive invocation on $(G_y,\partial_y)$.
  Without loss of generality, $G_x=G[A^y]$ where $(A^y,B^y)$ is a separation of $G_y-\partial_y$.  Then $\partial_x=(A^y\cap(\partial_y\cup B^y)=(A^y\cap (\partial_y\cup (A^y\cap B^y))=A^y\cap B_y$.  If $x$ is a leaf of $T$ then $B_x=V(G_x)=A^y$, so $B_x\cap B_y=A^y\cap B_y=\partial_x$.  If $x$ is not a leaf of $T$ then $B_x=\partial_x\cup (A^x\cap B^x)$ where $(A^x,B^x)$ is a separation of $G_x-\partial_x=G[A^y\setminus (A^y\cap B_y)]=G[A^y\setminus B_y]$. In particular, $A^x\cup B^x$ contains no vertex of $B_y$, so $B_x\cap B_y=(\partial_x\cup (A^x\cap B^x))\cap B_y= \partial_x\cap B_y=\partial_x$ since $\partial_x=A^y\cap B_y\subseteq B_y$.  In either case $\boundary_{\mathcal{T}}(x)=B_x\cap B_y=\partial_x$.

  Now the bounds on $|\interior_{\mathcal{T}}(x)|$ and $|\boundary_{\mathcal{T}}(x)|$ are easily established by induction on $d:=\depth_T(x)$: When $d=0$, $|\boundary_{\mathcal{T}}(x)|=|\emptyset|=0 = 0a$ and $|\interior_{\mathcal{T}}(x)|=|V(G)\setminus\boundary_{\mathcal{T}}(x)|=n=n\cdot (\tfrac{2}{3})^0$. Now consider a node $x$ of depth $d\ge 1$ with parent $y$ that was created by a recursive invocation on $(G_y,\partial_y)=(G_y,\boundary_\mathcal{T}(y))$. Without loss of generality $T_x$ was created by a recursive invocation on $(G[A^y],\partial_x)$ where $(A^y,B^y)$ is a balanced separation of $G_y-\boundary_y=\interior_{\mathcal{T}}(y)$. Then  $|\interior_{\mathcal{T}}(x)|\le \tfrac{2}{3}\cdot |\interior_{\mathcal{T}}(y)|\le \tfrac{2}{3}\cdot (\frac{2}{3})^{d-1}\cdot n=(\tfrac{2}{3})^d\cdot n$ and $|\boundary_{\mathcal{T}}(x)|=|\partial_x|\le |\partial_y|+|A^y\cap B^y|= |\boundary_{\mathcal{T}}(y)|+|A^y\cap B^y|\le (d-1)a + a=da$.

  The bound on $\height(T)$ follows from the bound $|\interior_{\mathcal{T}}(x)|\le n\cdot (\tfrac{2}{3})^{\depth_T(x)}$ and the fact that the algorithm returns a 1-node tree (of height $0$) when $|V(G')\setminus\partial|\le N=n\cdot(\tfrac{2}{3})^h$.
\end{proof}

\section{The Proof}

The following definition replaces the notion of $W$-clouds in \cite{dvorak.norin:treewidth}.
Let $W_{-1}:=\emptyset$, let $G$ be a graph, let $W\subseteq V(G)$, let $W_0\subseteq W_1\subseteq\cdots\subseteq W_{\ell}\subseteq W_{\ell+1}\subseteq V(G)$ be a nested sequence of vertex subsets of $G$, and let $\Delta_i:=W_{i}\setminus W_{i-1}$ and $s_i:=|\Delta_i|$, for each $i\in\{0,\ldots,\ell+1\}$.  Then $W_0,\ldots,W_{\ell+1}$ is a \defin{$W$-sequence of width $w$} in $G$ if it satisfies the following conditions:
\begin{enumerate}[nosep,nolistsep,label=\rm(\alph*),ref=(\alph*)]
  \item $W_0=W$.\label{w_starts}
  \item $s_i=w$, for each $i\in\{0,\ldots,\ell\}$.\label{uniform_size}
  \item $s_{\ell+1}\in\{0,\ldots,w-1\}$.\label{remainder}
  \item $G[W_i]$ contains $s_i$  pairwise vertex-disjoint $\Delta_i$-$W$ paths, for each $i\in\{0,\ldots,\ell+1\}$.\label{linked}
  \item There exists $Z\subseteq W_{\ell+1}$ with $|Z|=s_{\ell+1}$ that separates $V(G)\setminus W_{\ell}$ and $W$. \label{separated}
\end{enumerate}

\begin{lem}\label{w_sequence}
  For every graph $G$, every $W\subseteq V(G)$ and every non-negative integer $w\le |W|$, there exists a $W$-sequence of width $w$ in $G$.
\end{lem}

\begin{proof}
  Let $W_0:=W$ and suppose that sets $W_0,\ldots,W_{i}$ have been defined that satisfy \ref{w_starts}, \ref{uniform_size} and \ref{linked} for some $i\ge 0$.  (These conditions are trivially satisfied for $i=0$.)  We now show how to construct $W_{i+1}$.

  Let $w'$ be the maximum number of pairwise vertex-disjoint $(V(G)\setminus W_i)$-$W$ paths in $G$, let $r:=\min\{w,w'\}$, and let $P_1,\ldots,P_{r}$ be pairwise vertex-disjoint $(V(G)\setminus W_i)$-$W$ paths in $G$.  For each $j\in\{1,\ldots,r\}$ let $v_j$ be the last vertex of $P_j$ contained in $V(G)\setminus W_i$ and let $P_j'$ be the subpath of $P_j$ that begins at $v_j$ and ends at the first vertex of $P_j$ contained in $W$.  Set $W_{i+1}:=\{v_1,\ldots,v_r\}\cup W_i$, which implies that $s_{i+1}=r$.  The paths $P_1',\ldots,P_{s_{i+1}}'$ certify that $W_0,\ldots,W_{i+1}$ satisfies \ref{linked}.

  If $r=w$ then $W_{i+1}$ also satisfies \ref{uniform_size} so that we now a sequence $W_0,\ldots,W_{i+1}$ that satisfies \ref{w_starts}, \ref{uniform_size} and \ref{linked}.  In this case we can continue to define $W_{i+2}$ as above.

  If $r< w$ then we set $\ell:=i$. Since $r< w$, $s_{\ell+1}=r<w$, so this choice of $\ell$ satisfies \ref{remainder}.  Since $G$ does not contain $k:=r+1$ pairwise vertex-disjoint $(V(G)\setminus W_{i})$-$W$ paths, \cref{menger} implies that there exists $Z\subseteq V(G)$ with $|Z|\le r$ such that $G-Z$ has no $(V(G)\setminus W_{\ell})$-$W$ path.  Since the paths $P_1',\ldots,P_{r}'$ are pairwise vertex-disjoint  $(V(G)\setminus W_{\ell})$-$W$ paths, $Z$ must contain at least one vertex from each of these paths, so $|Z|\ge r$. Therefore, $r\le |Z|\le r$, so $|Z|=r=s_{\ell+1}$.  Since $V(P_j')\subseteq W_{\ell+1}$ for each $j\in\{1,\ldots,s_{\ell+1}\}$, $Z\subseteq W_{\ell+1}$.  Therefore, this choice of $Z$ satisfies \ref{separated}.  Therefore, the sequence $W_0,\ldots,W_{\ell+1}$ satisfies \ref{w_starts} (a condition on $W_0$), \ref{uniform_size} (conditions on $W_1,\ldots,W_{\ell}$), \ref{remainder} (a condition $W_{\ell+1}$), \ref{linked} (conditions on $W_0,\ldots,W_{\ell+1}$), and \ref{separated} (a condition on $W_{\ell+1}$).  Thus, $W_0,\ldots,W_{\ell+1}$ is a $W$-sequence of width $w$ in $G$.
\end{proof}

Observe that, for any $W$-sequence $W_0,\ldots,W_{\ell+1}$, we have the bounds $(\ell+1)|W|\le |W_{\ell+1}|< (\ell+2)|W|$, so $1/(\ell+2)< |W|/|W_{\ell+1}|\le 1/(\ell+1)$.
The following lemma shows that, for any separation $(A,B)$ of $G[W_{\ell+1}]$, the size of the intersection of $A\setminus B$ with $W\cup Z$ is bounded by the order $|A\cap B|$ of $(A,B)$ and the ratio $|A\setminus B|\cdot |W|/|W_{\ell+1}|$.

\begin{lem}\label{z_w_bound}
  Let $G$ be a graph, let $W\subseteq V(G)$, let $W_0,\ldots,W_{\ell+1}$ be a $W$-sequence of width $|W|$ in $G$ with $\ell\ge 1$, let $\Delta_{\ell+1}:=W_{\ell+1}\setminus W_{\ell}$,  and let $Z\subseteq W_{\ell+1}$, $|Z|=|\Delta_{\ell+1}|$, separate $V(G)\setminus W_{\ell}$ and $W$ in $G$.
  Then every $(V(G)\setminus W_{\ell},Z,W)$-separation $(X,Y)$ of $G$ has the property that, for any separation
  $(A,B)$ of $G[W_{\ell+1}]$,
  \[
    |W\setminus B|+|Z\setminus B|\le \frac{(2+\tfrac{1}{6})\,|A\setminus B|}{\ell+2}+3\,|A\cap B| \enspace .
  \]
\end{lem}

\begin{rem}\label{easier_path}
  The proof of \cref{z_w_bound} requires some extra effort to obtain $\ell+2$ in the denominator. The reader who is not interested in precise constants can already stop at \cref{z_b_bound} in the proof, from which the bound
  \[
    |W\setminus B|+|Z\setminus B|\le \frac{2|A\setminus B|}{\ell+1} + 3|A\cap B|
  \]
  follows immediately.
  This weaker bound is still sufficient to prove \cref{main_result} with the constant $c< 69$.
\end{rem}

\begin{proof}
  Let $W_{-1}:=\emptyset$ and, for each $i\in\{0,\ldots,\ell+1\}$, let $\Delta_{i}:=W_{i}\setminus W_{i-1}$ (as in the definition of $W$-sequence).
  By the definition of $W$-sequence, $G[W_{i}]$ contains a set $\mathcal{P}_i$ of $|W\setminus B|$ pairwise vertex-disjoint $\Delta_i$-$(W\setminus B)$ paths, for each $i\in\{0,\ldots,\ell\}$.

  We begin by bounding $|W\setminus B|$ using the path sets $\mathcal{P}_0,\ldots,\mathcal{P}_{\ell}$.  For each $i\in\{0,\ldots,\ell\}$, let $\mathcal{Q}_i\subseteq \mathcal{P}_i$ contain the paths in $\mathcal{P}_i$ that begin at a vertex in $A\setminus B$ and let $\overline{\mathcal{Q}}_i:=\mathcal{P}_i\setminus \mathcal{Q}_i$ contain the paths in $\mathcal{P}_i$ that begin at a vertex in $B$.  Since each path in $\mathcal{Q}_i$ begins at a distinct vertex in $\Delta_i\setminus B$, $|\mathcal{Q}_i|\le|\Delta_i\setminus B|$. Each path in $\overline{\mathcal{Q}}_i$ begins at a vertex in $B$ and ends at a vertex in $W\setminus B$. Since $(A,B)$ is a separation of $G[W_{\ell+1}]$, this implies that each path in $\overline{\mathcal{Q}}_i$ contains a vertex in $A\cap B$.  Since the paths in $\overline{\mathcal{Q}}_i$ are pairwise vertex-disjoint, $|\overline{\mathcal{Q}}_i|\le |A\cap B|$.

  Each vertex $w\in W\setminus B$ is the last vertex of exactly one path in $\mathcal{P}_i$, for each $i\in\{0,\ldots,\ell\}$. Thus, each vertex $w\in W\setminus B$ is the endpoint of exactly $\ell+1$ paths in $\bigcup_{i=0}^\ell\mathcal{P}_i$.  Since $\{\Delta_0,\ldots,\Delta_{\ell+1}\}$ is a partition of $W_{\ell+1}$, we have:
  \begin{align}
  \begin{split}
    |W\setminus B|
      & = \frac{1}{\ell+1}\cdot\sum_{i=0}^{\ell}|\mathcal{P}_i| \\
      & = \frac{1}{\ell+1}\cdot\sum_{i=0}^{\ell}\left(|\mathcal{Q}_i|+|\overline{\mathcal{Q}_i}|\right) \\
      & \le \frac{1}{\ell+1}\cdot\sum_{i=0}^{\ell}\left(|\Delta_i\setminus B|+|A\cap B|\right) \\
      & = \frac{|W_\ell\setminus B|}{\ell+1} + |A\cap B| \\
      & = \frac{|A\setminus B|-|\Delta_{\ell+1}\setminus B|}{\ell+1} + |A\cap B| \enspace .  \label{w_b_bound}
  \end{split}
  \end{align}

Next, we bound $|Z\setminus B|$.  By the definition of $W$-sequence, $G[W_{\ell+1}]$ has a set $\mathcal{R}$ of $|\Delta_{\ell+1}|=|Z|$ pairwise vertex-disjoint $\Delta_{\ell+1}$-$W$ paths.  Let $\mathcal{P}^\star:=\{P\in\mathcal{R}: V(P)\cap (Z\setminus B)\neq\emptyset\}$.
Since each of the paths in $\mathcal{R}$ contains a distinct vertex in $Z$, $|\mathcal{P}^\star|=|Z\setminus B|$.  Partition $\mathcal{P}^\star$ into three sets:

  \begin{enumerate}[nosep,nolistsep]\setcounter{enumi}{-1}
      \item $\mathcal{P}^\star_0$ are the paths in $\mathcal{P}^\star$ that start at a vertex of $\Delta_{\ell+1}\setminus B$ and end at a vertex in $W\setminus B$.
      \item $\mathcal{P}^\star_1$ are the paths in $\mathcal{P}^\star$ that start at a vertex in $\Delta_{\ell+1}\cap B$ and end at a vertex in $W\setminus B$.
      \item $\mathcal{P}^\star_2$ are the paths in $\mathcal{P}^\star$ that start at a vertex in $\Delta_{\ell+1}\setminus B$ and end at a vertex in $W\cap B$.
  \end{enumerate}
  Since the paths in $\mathcal{P}_0^\star\cup\mathcal{P}_1^\star$ are pairwise vertex-disjoint and each contains a vertex of $W\setminus B$, $|\mathcal{P}_0^\star|+|\mathcal{P}_1^\star|\le|W\setminus B|$.
  Each path in $\mathcal{P}_1^\star\cup \mathcal{P}_2^\star$ contains a vertex in $Z\setminus B$ and a vertex in $B$.  Since $(A,B)$ is a separation of $G[W_{\ell+1}]$ this implies that each path in $\mathcal{P}_1^\star\cup \mathcal{P}_2^\star$ contains a vertex in $A\cap B$. Since the paths in  $\mathcal{P}_1^\star\cup \mathcal{P}_2^\star$ are pairwise vertex-disjoint, this implies that $|\mathcal{P}_1^\star\cup \mathcal{P}_2^\star|\le |A\cap B|$. Therefore,
  \begin{equation}
     |Z\setminus B| = |\mathcal{P}^\star| = |\mathcal{P}_0^{\star}| + |\mathcal{P}_1^\star| + |\mathcal{P}_2^{\star}|
     \le |W\setminus B| + |\mathcal{P}_2^\star|
     \le \frac{|A\setminus B|-|\Delta_{\ell+1}\setminus B|}{\ell+1}+|A\cap B|+|\mathcal{P}_2^\star|
     \enspace , \label{coolio}
  \end{equation}
  where the last inequality is an application of inequality~\eqref{w_b_bound}.  At this point, adding \eqref{w_b_bound} and \eqref{coolio} and using the inequality $|\mathcal{P}_2^\star| \le |\mathcal{P}_1^\star\cup\mathcal{P}_2^\star|\le |A\cap B|$ immediately gives the bound discussed in \cref{easier_path}.  With a bit more work, we can do better.  Since each path in $\mathcal{P}_0^\star\cup\mathcal{P}^\star_2$ starts at a distinct vertex in $\Delta_{\ell+1}\setminus B$,
  \begin{equation}
      |\Delta_{\ell+1}\setminus B|\ge |\mathcal{P}_0^\star|+|\mathcal{P}^\star_2|=|Z\setminus B|-|\mathcal{P}_1^\star| \enspace . \label{z_b_trick}
  \end{equation}
  Using inequality~\eqref{z_b_trick}  in \cref{coolio} we obtain
  \[
     |Z\setminus B| \le \frac{|A\setminus B|-|Z\setminus B|+|\mathcal{P}_1^\star|}{\ell+1}+|A\cap B|+|\mathcal{P}_2^\star| \enspace .
  \]
  Rewriting this equation to isolate $|Z\setminus B|$, we obtain
  \begin{align}
     |Z\setminus B|
      & \le \frac{|A\setminus B|+|\mathcal{P}_1^\star|}{\ell+2} + \left(\frac{\ell+1}{\ell+2}\right)\cdot
        \left(|A\cap B|+|\mathcal{P}_2^\star|\right) \\
      & < \frac{|A\setminus B|}{\ell+2} +
        |A\cap B|+|\mathcal{P}_2^\star|+\tfrac{1}{3} |\mathcal{P}_1^\star|
        \enspace . \label{z_b_bound}
  \end{align}
  (The second inequality uses the assumption that $\ell\ge 1$.) Using inequality~\eqref{z_b_trick}  in \cref{w_b_bound} we obtain,
  \begin{equation}
      |W\setminus B|+\frac{|Z\setminus B|}{\ell+1}\le \frac{|A\setminus B|+|\mathcal{P}_1^\star|}{\ell+1} + |A\cap B|
      \le \frac{|A\setminus B|}{\ell+1} + |A\cap B| +\tfrac{1}{2}|\mathcal{P}_1^\star| \enspace .  \label{w_b_bound_two}
  \end{equation}
  (The second inequality again uses the assumption that $\ell\ge 1$.) Adding \eqref{w_b_bound_two} and \eqref{z_b_bound} and using the fact that $\tfrac{5}{6}|\mathcal{P}_1^\star|+|\mathcal{P}_2^\star|\le|\mathcal{P}_1^\star|+|\mathcal{P}_2^\star|\le |A\cap B|$, we obtain
  \begin{align}
    |W\setminus B| + \left(\frac{\ell+2}{\ell+1}\right)\cdot|Z\setminus B|
    & \le \frac{|A\setminus B|}{\ell+1} +
    \frac{|A\setminus B|}{\ell+2} + 3\,|A\cap B| \label{sum_bound}
  \end{align}
  We can now upper bound $|W\setminus B|+|Z\setminus B|$ by maximizing $x_0+x_1$ subject to
  \begin{equation}
     x_0+\left(\frac{\ell+2}{\ell+1}\right)\cdot x_1 \le R \enspace , \label{sum_bound_two}
  \end{equation}
  where $R$ denotes the expression in \eqref{sum_bound}.
  For a fixed $x_0=x^\star_0$ the maximum value of $x_1$ that satisfies \eqref{sum_bound_two} is
  \[
    x_1 = x_1^\star:=\left(\frac{\ell+1}{\ell+2}\right)\cdot(R-x^{\star}_0)
  \]
  in which case
  \[
     x_0 + x_1 = x^\star_0+x^\star_1 = \left(\frac{\ell+1}{\ell+2}\right)\cdot R + \frac{x^\star_0}{\ell+2}
  \]
  maximizes $x_0+x_1$ subject to fixed $x_0=x_0^{\star}$.
  This is an increasing linear function of $x_0^{\star}$.  From \eqref{w_b_bound}, we have the constraint
  \[
     x_0^\star \le
     \frac{|A\setminus B|-|\Delta_{\ell+1}\setminus B|}{\ell+1}+|A\cap B|
     \le \frac{|A\setminus B|}{\ell+1}+|A\cap B|
  \]
  from which we obtain the upper bound
  \begin{align*}
    |W\setminus B|+|Z\setminus B|
    & \le \left(\frac{\ell+1}{\ell+2}\right)\cdot R +
    \frac{x_0^\star}{\ell+2} \\
    & \le \left(\frac{\ell+1}{\ell+2}\right)\left(\frac{|A\setminus B|}{\ell+1}+\frac{|A\setminus B|}{\ell+2}+3\,|A\cap B|\right)+\frac{|A\setminus B|-|Z\setminus B|}{(\ell+1)
    (\ell+2)}
    + \frac{|A\cap B|}{\ell+2}
    \\
    & = \frac{|A\setminus B|}{\ell+2}\cdot\left(\frac{\ell+1}{\ell+1}+\frac{\ell+1}{\ell+2}+\frac{1}{\ell+1}\right) +\left(\frac{3\ell+4}{\ell+2}\right)\cdot |A\cap B|  \\
    & = \frac{|A\setminus B|}{\ell+2}\cdot\left(2-\frac{1}{\ell+2}+\frac{1}{\ell+1}\right) +\left(\frac{3\ell+4}{\ell+2}\right)\cdot |A\cap B|  \\
    & = \frac{|A\setminus B|}{\ell+2}\cdot\left(2+\frac{1}{(\ell+1)(\ell+2)}\right) +\left(\frac{3\ell+4}{\ell+2}\right)\cdot |A\cap B|  \\
    & < \frac{(2+\tfrac{1}{6})\,|A\setminus B|}{\ell+2} + 3\,|A\cap B| \enspace . \qedhere
  \end{align*}
\end{proof}

We are now ready to prove \cref{main_result}.

\begin{proof}[Proof of \cref{main_result}]
  Let $G$ be a graph and let $a:=\sep(G)$.  Let
  \[
    h:=4 \quad\text{and}\quad t:=\frac{4h}{1-(2+\tfrac{1}{6})\cdot(\tfrac{2}{3})^{h}} = \frac{3888}{139} < 27.972 \enspace .
  \]
  We will show that $\tw(G)< (2t+1)a$. The proof is by induction on the number of vertices of $G$. We will prove the following stronger statement: For any graph $G$ and any non-empty subset $W\subseteq V(G)$ of size at most $ta$, $G$ has a tree decomposition $(B_x:x\in V(T))$ of width less than $(2t+1)a$ in which $W\subseteq B_x$ for some $x\in V(T)$.

  If $G$ has less than $ta$ vertices, then the proof is trivial. We use a tree $T$ with a single vertex $x$ and set $B_x:=V(G)$.  We now assume that $|V(G)|\ge ta$.
  By \cref{w_sequence}, $G$ has a $W$-sequence $W_0,\ldots,W_{\ell+1}$ of width $|W|$.   Let $\Delta_0,\ldots,\Delta_{\ell+1}$ and $Z$ be as in the definition of $W$-sequence.  Let $(X,Y)$ be a $(V(G)\setminus W_{\ell},Z,W)$-separation.

  Since $|X\cap Y|=|Z|<|W|\le ta$, the inductive hypothesis implies that $G[X]$ has a tree decomposition $\mathcal{T}_X:=(B_x:x\in V(T_X))$ of width less than $(2t+1)a$ in which $Z\subseteq B_x$ for some $x \in V(T_X)$. To finish the proof, we construct a tree decomposition $\mathcal{T}_Y:=(B_y:y\in V(T_Y))$ of $G[Y]$ of width less than $(2t+1)a$ in which some bag $B_{y}$ contains $W\cup Z$.  Then the tree $T$ obtained by joining $T_Y$ and $T_X$ using the edge $x y$ gives the desired tree decomposition $\mathcal{T}:=(B_x:x\in V(T))$.

  If $\ell=0$ then we use the trivial tree decomposition in which $T_Y$ has a single node $y$ where $B_{y}:=W\cup Z$.  Since $W\cup Z\subseteq Y$ and $|Y|=|W_1|=|W|+|Z|< 2|W|<(2t+1)a$, this decomposition has width less than $(2t+1)a$. We now assume that $\ell \ge  1$.

  Since $\ell \ge 1$, $W_0,\ldots,W_{\ell+1}$ satisfies the conditions of \cref{z_w_bound}.
  Let $\mathcal{T}'_Y:=(B'_y:y\in V(T'_Y))$ be the tree decomposition of $G[W_{\ell+1}]$ obtained by applying \cref{separation_tree} to $G[W_{\ell+1}]$ with the height $h$ defined above.  The following claim will be used to bound the width of a tree decomposition that we derive from $\mathcal{T}'_Y$.

  \begin{clm} \label{cell_bound}
     For each $d\in\{0,\ldots,h\}$ and
     each node $y$ of $T'_y$ with $|\interior_{\mathcal{T}'_Y}(y)|\le (\tfrac{2}{3})^d\cdot |W_{\ell+1}|$,
     \begin{equation}
       |\interior_{\mathcal{T}'_Y}(y)\cap (W\cup Z)|  \le (2+\tfrac{1}{6})ta\cdot (\tfrac{2}{3})^{d}+3da \label{t_bound}
     \end{equation}
   \end{clm}

  \begin{clmproof}
  Consider the separation
  \[
    (A,B):=(\interior_{\mathcal{T}'_Y}(y)\cup\boundary_{\mathcal{T}'_Y}(y),W_{\ell+1}\setminus\interior_{\mathcal{T}'_Y}(y))
  \]
  of $G[W_{\ell+1}]$ and observe that $A\setminus B:=\interior_{\mathcal{T}'_Y}(y)$.
  By \cref{z_w_bound},

  \begin{align*}
  |\interior_{\mathcal{T}'_Y}(y)\cap (W\cup Z)|
  & = |(A\setminus B)\cap (W\cup Z)| \\
  & \le |W\setminus B| + |Z\setminus B| \\
  & \le
     \frac{(2+\tfrac{1}{6})\,|A\setminus B|}{\ell+2} + 3\,|A\cap B| \\
  & \le \frac{(2+\tfrac{1}{6})\,(\tfrac{2}{3})^d|W_{\ell+1}|}{\ell+2} + 3\,|A\cap B| \\
  & = \frac{(2+\tfrac{1}{6})\,(\tfrac{2}{3})^d((\ell+1)|W|+|\Delta_{\ell+1}|)}{\ell+2} + 3\,|A\cap B| \\
  & < \frac{(2+\tfrac{1}{6})\,(\tfrac{2}{3})^d(\ell+2)|W|}{\ell+2} + 3\,|A\cap B| \\
  & \le (2+\tfrac{1}{6})ta\cdot(\tfrac{2}{3})^d + 3da \enspace . \qedhere
  \end{align*}
  \end{clmproof}

 For each node $y$ of $T'_y$, define $B_y:=(B'_y\cap Y) \cup (\interior_{\mathcal{T}'_Y}(y)\cap (W\cup Z))$ and let $\mathcal{T}''_Y:=(B_y:y\in V(T'_y))$.  By \cref{restricted_decomp} (applied with the separation $(X\cup W,Y\cup W)$), $\mathcal{T}''_Y$ is a tree decomposition of $G[Y]$ in which the root bag contains $(X\cup W)\cap (Y\cup W)=Z\cup W$.

 \begin{clm}\label{leaf_interface}
    For each leaf $y$ of $T'_Y$, $|\boundary_{\mathcal{T}''_Y}(y)|\le ta$.
 \end{clm}
 \begin{proof}
   Let $y$ be a leaf of $T'_Y$ and let $z$ be the parent of $y$.  Then, by \cref{separation_tree}, $|\interior_{\mathcal{T}'_Y}(y)|\le (\tfrac{2}{3})^h\cdot |W_{\ell+1}|$ and $\boundary_{\mathcal{T}_Y''}(x)=B_y\cap B_z\subseteq (B'_y\cap B'_z)\cup(\interior_{\mathcal{T}'_Y}(y)\cap (W\cup Z)) = \boundary_{\mathcal{T}_Y'}(y)\cup(\interior_{\mathcal{T}'_Y}(y)\cap (W\cup Z))$. Therefore, by \cref{separation_tree} and \cref{cell_bound}
   \[
     |\boundary_{\mathcal{T}''_Y}(y)| \le |\boundary_{\mathcal{T}'_Y}(y)| + |\interior_{\mathcal{T}'_Y}(y)\cap (W\cup Z)| \le ha + (2+\tfrac{1}{6})ta\cdot(\tfrac{2}{3})^{h} + 3ha =  (2+\tfrac{1}{6})ta\cdot(\tfrac{2}{3})^{h} + 4ha\enspace .
   \]
   The values of $h$ and $t$, defined above, are chosen so that the right hand side of this inequality is equal to $ta$.
 \end{proof}

  By \cref{leaf_interface}, $|\boundary_{\mathcal{T}''_Y}(y)|\le ta$ for each leaf $y$ of $T_Y'$. Therefore, by the inductive hypothesis, $G[\interior_{\mathcal{T}''_Y}(y)\cup\boundary_{\mathcal{T}''_Y}(y)]$ has a tree decomposition $\mathcal{T}^y:=(B_y:y\in V(T^y))$ of width less than $(2t+1)a$ in which some bag $B_{y_0}$ contains  $\boundary_{\mathcal{T}''_Y}(y)$, for each leaf $y$ of $T_Y'$. Create a new tree $T_Y$ from $T_Y'$ by replacing each leaf $y$ of $T_Y'$ with the node $y_0$ from the tree $T^y$.  Then $\mathcal{T}_Y:=(B_y:y\in V(T_Y))$ is a tree decomposition of $G[Y]$.
  \begin{clm}\label{treewidth_bound}
     The width of $\mathcal{T}_Y$ is less than $(2t+1)a$.
  \end{clm}

 \begin{clmproof}
    The inductive hypothesis ensures that all bags of the tree decomposition have size at most $(2t+1)a$ except for those associated with non-leaf nodes of $T'_Y$.
    Let $y$ be a non-leaf node in $T_Y'$ whose depth is $d < h$. If $d = 0$, then
    \[
      |B_y| \le |W\cup Z| + |B'_y| \le (2ta-1)+a < (2t+1)a \enspace .
    \]
    If $d\ge 1$ then, by \cref{separation_tree}, $|B'_y|\le (d+1)a$ and $|\interior_{\mathcal{T}_Y'}(y)|\le |W_\ell|\cdot(\tfrac{2}{3})^{d}$.
    By \cref{cell_bound},
    \[
      |B_y| \le |\interior_{\mathcal{T}_Y'}(y)\cap (W\cup Z)| + |B'_y|
      \le (2+\tfrac{1}{6})ta  \cdot (\tfrac{2}{3})^{d}+3da + (d+1)a
      \le (2+\tfrac{1}{6})ta\cdot (\tfrac{2}{3})^{d}+(4d+1)a \enspace .
    \]
    With the choices of $t$ and $h$ above, the right hand side of this equation is less than $(2t+1)a$ for all $d\in\{1,\ldots,h-1\}$.
  \end{clmproof}
  Let $y_r$ be the root of $T_Y$. Then, $B_{y_r}:=B'_{y_r}\cap Y\supseteq W\cup Z$.  Therefore, $\mathcal{T}_Y=(B_y:y\in V(T_Y))$ is a tree decomposition of $G[Y]$ that (by \cref{treewidth_bound}) has width less than $(2t+1)a$ and there exists $y\in V(T_Y)$ such that $W\cup Z\subseteq B_{y}$.  This completes the proof.
\end{proof}

\bibliographystyle{plainurlnat}
\bibliography{dnr}

\end{document}